\documentclass[11pt]{article}
\usepackage{amssymb}
\usepackage{latexsym}
\usepackage{amsmath}
\usepackage{amsthm}
\usepackage{amsfonts}
\usepackage{color}
\usepackage{pdfsync}
\usepackage[usenames,dvipsnames]{pstricks}
\usepackage{epsfig}

\newcommand{\R}{\mathbb{R}}

\numberwithin{equation}{section}

\usepackage{amssymb}

\newcommand{\beq}{\begin{equation} }
\newcommand{\eqq}{\end{equation} }
\newcommand{\cuad}{{\sqcap\kern-.68em\sqcup}}

\newtheorem{lemma}{Lemma}[section]

\newtheorem{remark}{Remark}[section]
\newcommand{\bremark}{\begin{remark} \em}
\newcommand{\eremark}{\end{remark} }

\def\beeq{\begin{equation}}
\def\eeq{\end{equation}}
\newcommand{\begeqaet}{\begin{eqnarray*}}
\newcommand{\eneqaet}{\end{eqnarray*}}

\let\Section=\section
\def\section{\setcounter{equation}{0}\Section}

\newtheorem{Thm}{Theorem}[section]

\newtheorem{Prop}{Proposition}[section]
\newtheorem{Remark}{Remark}[section]

\begin{document}
\begin{center}{\bf\Large Concentration of ground state solution for a fractional Hamiltonian Systems }\medskip
%%%%%%%%%%%%%%%%%%%%%%%%%%%%%%%%%%%%%%%%%%%%%%%%%%%%%%%%%%%%%%%%%%%%%%
%%%%%%%%%%%%%%%%%%%%%%%%%%%%%%%%%%%%%%%%%%%%%%%%%%%%%%%%%%%%%%%%%%%%%%

\bigskip

\bigskip

{C\'esar E. Torres Ledesma}

Departamento de Matem\'aticas, \\
Universidad Nacional de Trujillo,\\
Av. Juan Pablo II s/n. Trujillo-Per\'u\\
{\sl  ctl\_576@yahoo.es}

\noindent

{Ziheng Zhang}

Department of Mathematics,\\ 
Tianjin Polytechnic University,\\ 
Tianjin 300387, China.

{\sl zhzh@mail.bnu.edu.cn}

%%%%%%%%%%%%%%%%%%%%%%%%%%%%%%%%%%%%%%%%%%%%%%%%%%%%%%%%%%%%%%%%%%%%%%
%%%%%%%%%%%%%%%%%%%%%%%%%%%%%%%%%%%%%%%%%%%%%%%%%%%%%%%%%%%%%%%%%%%%%%
%\keywords{}
%\subjclass{}

\medskip

\medskip
\medskip
\medskip
\medskip

\end{center}
%%%%%%%%%%%%%%%%%%%%%%%%%%%%%%%%%%%%%%%%%%%%%%%%%%%%%%%%%%%%%%%%%%%%%%

\centerline{\bf Abstract}

\medskip

In this paper we are concerned with the existence of ground states solutions for the following fractional Hamiltonian systems 
$$
\left\{
  \begin{array}{ll}
   - _tD^{\alpha}_{\infty}(_{-\infty}D^{\alpha}_{t}u(t))-\lambda L(t)u(t)+\nabla W(t,u(t))=0,\\[0.1cm]
    u\in H^{\alpha}(\mathbb{R},\mathbb{R}^n),
  \end{array}
\right.
  \eqno(\mbox{FHS})_\lambda
$$
where $\alpha\in (1/2,1)$, $t\in \mathbb{R}$, $u\in \mathbb{R}^n$, $\lambda>0$ is a parameter, $L\in C(\mathbb{R},\mathbb{R}^{n^2})$ is a symmetric matrix
for all $t\in \mathbb{R}$, $W\in C^1(\mathbb{R} \times \mathbb{R}^n,\mathbb{R})$ and $\nabla W(t,u)$ is the gradient of $W(t,u)$ at $u$. Assuming that $L(t)$ is a positive semi-definite symmetric matrix for all $t\in \mathbb{R}$, that is, $L(t)\equiv 0$ is allowed to occur in some finite interval $T$ of $\mathbb{R}$,
$W(t,u)$ satisfies Ambrosetti-Rabinowitz condition and some other reasonable hypotheses, we show that (FHS)$_\lambda$ has a ground sate solution which vanishes on $\mathbb{R}\setminus T$ as $\lambda \to \infty$, and converges to $u\in H^{\alpha}(\R, \R^n)$, where $u\in E_{0}^{\alpha}$ is a ground state solution of the Dirichlet BVP for fractional systems on the finite interval $T$. Recent results are generalized and significantly improved.


\medskip

%%%%%%%%%%%%%%%%%%%%%%%%%%%%%%%%%%%%%%%%%%%%%%%%%%%%%%%%%%%%%%%%%%%%%%
\date{}
%\maketitle
%%%%%%%%%%%%%%%%%%%%%%%%%%%%%%%%%%%%%%%%%%%%%%%%%%%%%%%%%%%%%%%%%%%%%%
%%%%%%%%%%%%%%%%%%%%%%%%%%%%%%%%%%%%%%%%%%%%%%%%%%%%%%%%%%%%%%%%%%%%%%

\setcounter{equation}{0}
\section{ Introduction}

Fractional differential equations both ordinary and partial ones are applied in mathematical modeling of process in physics,
mechanics, control theory, biochemistry, bioengineering and economics. Therefore, the theory of fractional differential
equations is an area intensively developed during the last decades \cite{ATMS04,Hilfer00}. The monographs
\cite{KST06,MR93,Pod99} enclose a review of methods of solving fractional differential equations, which are an extension of procedures from differential equations theory.

Recently, also equations including both left and right fractional derivatives are discussed.
Apart from their possible applications, equations with left and right derivatives is an interesting and new field in
fractional differential equations theory. In this topic, many results are obtained dealing with the existence and multiplicity of solutions of nonlinear fractional
differential equations by using techniques of nonlinear analysis, such as fixed point theory (including Leray-Schauder
nonlinear alternative) \cite{BaiLu05}, topological degree theory (including co-incidence degree theory) \cite{Jiang11}
and comparison method (including upper and lower solutions and monotone iterative method) \cite{Zhang11} and so on.

It should be noted that critical point theory and variational methods have also turned out to be very effective tools
in determining the existence of solutions for integer order differential equations. The idea behind them is trying to
find solutions of a given boundary value problem by looking for critical points of a suitable energy functional defined
on an appropriate function space. In the last 30 years, the critical point theory has become a wonderful tool in
studying the existence of solutions to differential equations with variational structures, we refer the reader to the
books due to Mawhin and Willem \cite{MawhinW89}, Rabinowitz \cite{Rab86}, Schechter \cite{Schechter99} and the references
listed therein.

In (FHS)$_\lambda$, if $\alpha=1$ and $\lambda=1$, then it reduces to the following second order Hamiltonian systems
$$
\ddot u-L(t) u+\nabla W(t,u)=0.\eqno(\mbox{HS})
$$
It is well known that the existence of homoclinic solutions for Hamiltonian systems and their importance in the study of
the behavior of dynamical systems have been recognized from Poincar\'{e} \cite{Poincare}. They may be ``organizing centers"
for the dynamics in their neighborhood. From their existence one may, under certain conditions, infer the existence of
chaos nearby or the bifurcation behavior of periodic orbits. During the past two decades, with the works of \cite{Omana92} and
\cite{Rab90} variational methods and critical point theory have been successfully applied for the search of the existence
and multiplicity of homoclinic solutions of (HS).

Assuming that $L(t)$ and $W(t,u)$ are independent of $t$ or periodic in $t$, many authors have studied the existence of homoclinic solutions for (HS), see for instance \cite{Co91,Ding95,Rab90} and the references therein and some more
general Hamiltonian systems are considered in the recent papers \cite{Izydorek05,Izydorek07}. In this case, the existence of
homoclinic solutions can be obtained by going to the limit of periodic solutions of approximating problems.
If $L(t)$ and $W(t,u)$ are neither autonomous nor periodic in $t$, the existence of homoclinic solutions of (HS) is quite different from the periodic
systems, because of the lack of compactness of the Sobolev embedding, such as \cite{Ding95,Omana92,Rab91} and the references mentioned there.

Motivated by the above classical works, in \cite{Torres12} the author considered the following fractional Hamiltonian systems
$$
\left\{
  \begin{array}{ll}
    _tD^{\alpha}_{\infty}(_{-\infty}D^{\alpha}_{t}u(t))+L(t)u(t)=\nabla W(t,u(t)),\\[0.1cm]
    u\in H^{\alpha}(\mathbb{R},\mathbb{R}^n),
  \end{array}
\right.
  \eqno(\mbox{FHS})
$$
where $\alpha\in (1/2,1)$, $t\in \mathbb{R}$, $u\in \mathbb{R}^n$, $L\in C(\mathbb{R},\mathbb{R}^{n^2})$ is a symmetric and positive definite matrix
for all $t\in \mathbb{R}$, $W\in C^1(\mathbb{R}\times \mathbb{R}^n,\mathbb{R})$ and $\nabla W(t,u)$ is the gradient of $W(t,u)$ at $u$. Assuming that $L(t)$ and
$W(t,u)$ satisfy the following hypotheses, the author showed that (FHS) possesses at least one nontrivial solution via Mountain
Pass Theorem. Explicitly,
\begin{itemize}
\item[(L)] $L(t)$ is a positive definite symmetric matrix for all $t\in \mathbb{R}$ and there exists an $l\in C(\mathbb{R},(0,\infty))$ such
that $l(t)\rightarrow \infty$ as $|t|\rightarrow \infty$ and
\begin{equation}\label{eqn:L coercive}
(L(t)u,u)\geq l(t)|u|^2 \quad \mbox{for all}\,\, t\in \mathbb{R} \,\, \mbox{and} \,\, u\in \mathbb{R}^n.
\end{equation}
\item[(W$_1$)]$W\in C^1(\mathbb{R} \times \mathbb{R}^n,\mathbb{R})$ and there is a constant $\theta>2$ such that
$$
0<\theta W(t,u)\leq (\nabla W(t,u),u)\quad \mbox{for all}\,\, t\in \mathbb{R} \,\,\mbox{and}\,\, u\in \mathbb{R}^n\backslash\{0\}.
$$
\item[(W$_2$)]$|\nabla W(t,u)|=o(|u|)$ as $|u|\rightarrow 0$ uniformly with respect to $t\in \mathbb{R}$.
\item[(W$_3$)] There exists $\overline{W}\in C(\mathbb{R}^n,\mathbb{R})$ such that
$$
|W(t,u)|+|\nabla W(t,u)|\leq |\overline W(u)|\quad \mbox{for every} \,\, t\in \mathbb{R}\,\, \mbox{and}\,\, u\in \mathbb{R}^n.
$$
\end{itemize}
(W$_1$) is the so-called global Ambrosetti-Rabinowitz condition, which implies that $W(t,u)$ is of superquadratic growth as $|u|\rightarrow \infty$.
Inspired by this work, using the genus properties of critical point theory, in \cite{ZhangYuan} the authors established some new criterion to guarantee
the existence of infinitely many solutions of (FHS) for the case that $W(t,u)$ is subquadratic as $|u|\rightarrow \infty$, where the condition (L) is also
needed to guarantee that the functional corresponding to (FHS) satisfies (PS) condition (see \cite{MendezTorres15} where a similar result was obtained). In addition,
very recently, using the fountain theorem, in \cite{XuReganZhang15}, the authors established the existence of infinitely many solutions of (FHS) for the case that
$W(t,u)$ is superquadratic as $|u|\rightarrow \infty$ without the Ambrosetti-Rabinowitz condition. Moreover, recently in  \cite{Torres14} the author firstly discussed the following
perturbed fractional Hamiltonian systems
$$
\left\{
  \begin{array}{ll}
   - _tD^{\alpha}_{\infty}(_{-\infty}D^{\alpha}_{t}u(t))-L(t)u(t)+\nabla W(t,u(t))=f(t),\\[0.1cm]
    u\in H^{\alpha}(\mathbb{R},\mathbb{R}^n),
  \end{array}
\right.
  \eqno(\mbox{PFHS})
$$
where $\alpha\in (1/2,1)$, $t\in \mathbb{R}$, $u\in \mathbb{R}^n$, $L\in C(\mathbb{R},\mathbb{R}^{n^2})$ is a symmetric and positive definite matrix
for all $t\in \mathbb{R}$, $W\in C^1(\mathbb{R} \times \mathbb{R}^n,\mathbb{R})$ and $\nabla W(t,u)$ is the gradient of $W(t,u)$ at $u$, $f\in C(\mathbb{R}, \mathbb{R}^n)$ and belongs to $L^2(\mathbb{R},\mathbb{R}^n)$.
Under the conditions of (L), (W$_1$)-(W$_3$) and assuming that the $L^2$ norm of $f$ is sufficiently small, he showed that (PFHS) has at least two
nontrivial solutions, which has been generalized in \cite{XuReganZhang15} where the condition (L) is also satisfied.

As is well-known, the condition (L) is the so-called coercive condition and is very restrictive. In fact, for a simple choice like $L(t)=\tau Id_n$, the condition (\ref{eqn:L coercive})
is not satisfied, where $\tau>0$ and $Id_n$ is the $n\times n$ identity matrix. Motivated by this point, in \cite{ZhangYuan14} the authors focused their attentions on the case that
$L(t)$ is bounded in the sense that
\begin{itemize}
\item[(L)$'$] $L\in C(\mathbb{R},\mathbb{R}^{n^2})$ is a symmetric and positive definite matrix for all $t\in \mathbb{R}$ and there are constants $0<\tau_1<\tau_2<\infty$ such that
$$
\tau_1|u|^2\leq (L(t)u,u)\leq \tau_2|u|^2\quad \mbox{for all}\,\, (t,u)\in \mathbb{R} \times \mathbb{R}^n.
$$
\end{itemize}
If the potential $W(t,u)$ is supposed to be subquadratic as $|u|\rightarrow +\infty$, then they also showed that (FHS) possessed infinitely many solutions.

Here we must point out, to obtain the existence or multiplicity of solutions for (FHS) (or (PFHS)), all the papers mentioned above need the assumption that
the symmetric matrix $L(t)$ is positive definite, see (L) and (L)$'$. Inspired by \cite{Torres12,Torres14,ZhangYuan,ZhangYuan14}, in present paper we deal with the following fractional Hamiltonian systems with a parameter
$$
\left\{
  \begin{array}{ll}
   - _tD^{\alpha}_{\infty}(_{-\infty}D^{\alpha}_{t}u(t))-\lambda L(t)u(t)+\nabla W(t,u(t))=0,\\[0.1cm]
    u\in H^{\alpha}(\mathbb{R},\mathbb{R}^n),
  \end{array}
\right.
  \eqno(\mbox{FHS})_{\lambda}
$$
where $\alpha\in (1/2,1)$, $t\in \mathbb{R}$, $u\in \mathbb{R}^n$, $\lambda>0$ is a parameter, $L\in C(\mathbb{R},\mathbb{R}^{n^2})$ is a symmetric matrix
for all $t\in \mathbb{R}$, $W\in C^1(\mathbb{R} \times \mathbb{R}^n,\mathbb{R})$ and $\nabla W(t,u)$ is the gradient of $W(t,u)$ at $u$. Unlike the papers on this problem, we require that
$L(t)$ is a positive semi-definite symmetric matrix for all $t\in \mathbb{R}$, that is, $L(t)\equiv 0$ is allowed to occur in some finite interval $T$ of $\mathbb{R}$. Explicitly,
\begin{itemize}
\item[$(\mathcal{L})_1$]$L\in C(\mathbb{R},\mathbb{R}^{n\times n})$ is a symmetric matrix for all $t\in\mathbb{R}$; there exists a nonnegative continuous function $l(t):\mathbb{R} \rightarrow \mathbb{R}$ and
a constant $c>0$ such that
$$
(L(t)u,u)\geq l(t)|u|^2,
$$
and the set $\{l<c\}:=\{t\in \mathbb{R} \,|\,l(t)<c\}$ is nonempty with $meas \{l<c\}<\frac{1}{C_{\infty}^2}$, where $meas \{\cdot\}$ is the Lebesgue measure and $C_\infty$
is the best Sobolev constant for the embedding of $X^{\alpha}$ into $L^{\infty}(\mathbb{R})$;
\item[$(\mathcal{L})_2$]$J=int (l^{-1}(0))$ is a nonempty finite interval and $\overline{J}=l^{-1}(0)$;
\item[$(\mathcal{L})_3$]there exists an open interval $T\subset J$ such that $L(t)\equiv 0$ for all $t\in \overline{T}$.
\end{itemize}
In this case, we assume that $W \in C^{1}(\mathbb{R} \times \mathbb{R}^n, \mathbb{R}^n)$ satisfy $(W_1)-(W_3)$  and:
\begin{itemize}
\item[$(W_4)$] $s\to \frac{\langle \nabla W(t,sq) , q\rangle}{s^{\theta-1}}$ is strictly increasing for all $q\neq 0$ and $s>0$, $\theta$ is given by ($W_1$). 
\end{itemize}

\begin{Remark}\label{nta1}
We note that, under the assumption of $(W_1)$, there are constants $c_1>0$ and $c_2>0$  such that (see \cite{Torres12}): 
\begin{enumerate}
\item[(i)] $W(t,u) \geq c_1|u|^\theta$, $|u|\geq 1$,
\item[(ii)] $W(t,u) \leq c_2|u|^\theta$, $|u|\leq 1$. 
\end{enumerate}
Furthermore, by ($i$) we obtain that
\begin{equation}\label{eqif}
\lim_{|u|\to \infty} \frac{W(t,u)}{|u|^2} = \infty,\;\;\mbox{uniformly in}\;\;t.
\end{equation}
Since $W(t,q)$ must be replaced by $W(t,q) - W(t,0)$, we may also assume without loss of generality that $W(t,0) = 0$ for all $t$. 
\end{Remark}
Now, we are in the position to state our main result.

\begin{Thm}\label{Thm:MainTheorem1}
Suppose that {\rm ($\mathcal{L}$)$_1$}-{\rm ($\mathcal{L}$)$_3$}, $(W_1) - (W_4)$ are satisfied, then there exists $\Lambda _*>0$ such that for every
$\lambda>\Lambda_*$, {\rm(FHS)$_\lambda$} has a ground state solution.
\end{Thm}

\begin{Remark}\label{nta2}
{\rm 
Note that in ($\mathcal{L}$)$_1$-($\mathcal{L}$)$_3$, we assume that $L(t)$ is a positive semi-definite symmetric matrix for all $t\in\mathbb{R}$. Therefore, the hypothesis (L) and (L)$'$
on $L(t)$ are not satisfied. Thus the results in \cite{Torres12,Torres14,ZhangYuan,ZhangYuan14} are generalized and improved significantly.

Moreover, as mentioned above, the coercive condition (L) is used to establish some compact embedding theorems to guarantee that (PS) condition (or the other weak compactness conditions) holds, which is the essential
step to obtain the existence of homoclinic solutions of (FHS) via Mountain Pass Theorem. In present paper, we assume that $L(t)$ satisfies ($\mathcal{L}$)$_1$-($\mathcal{L}$)$_3$
and could not obtain some compact embedding theorem.
Therefore, the main difficulty is to adapt some new technique to overcome this difficulty and test the (PS) condition is verified, see Lemma \ref{TMlem5}.}
\end{Remark}

Here we must mention the recent works \cite{Torres15}, \cite{ZhangTorres}.  In fact, in \cite{Torres15}, assuming that $L(t)$ satisfies ($\mathcal{L}$)$_1$-($\mathcal{L}$)$_3$, then the author
showed that (FHS)$_\lambda$ has at least one nontrivial solution for the case that the potential $W(t,u)$ satisfies the following subquadratic assumptions as
$|u|\rightarrow \infty$:
\begin{itemize}
\item[(W$_5$)]there exist a constant $\gamma\in (1,2)$ and a positive function $b\in L^p (\mathbb{R})$ with $p\in (1,\frac{2}{2-\gamma}]$ such that
$$
|\nabla W(t,u)|\leq b(t)|u|^{\gamma-1}\quad \mbox{for all}\,\, (t,u\in \mathbb{R}\times \mathbb{R}^n);
$$
\item[(W$_6$)]there exist two constants $\eta$, $\delta>0$ such that
$$
|W(t,u)|\geq \eta |u|^\gamma \quad \mbox{for all} \, \, x\in T \,\, \mbox{and} \,\, u\in \mathbb{R}\,\, \mbox{with} \,\, |u|\leq \delta.
$$
\end{itemize}
$|u|\rightarrow +\infty$. Furthermore in \cite{ZhangTorres}, the authors have complemented the previous work by consider superquadratic potential when $|u| \to \infty$. They obtain the same results as in \cite{Torres15}.

For technical reason, we consider that there exists $0<L< +\infty$, such that $T = [0,L]$, where $T$ is given by $(\mathcal{L})_3$. On the concentration of solutions we have the following result.
\begin{Thm}\label{Thm:MainTheorem2}
Let $u_\lambda$ be  a solution of problem $(FHS)_\lambda$ obtained in Theorem \ref{Thm:MainTheorem1}, then $u_\lambda \to \tilde{u}$ strongly in $H^{\alpha}(\mathbb{R})$ as $\lambda \to \infty$, where $\tilde{u}$ is a ground state solution of the equation
\begin{eqnarray}\label{08}
&{_{t}}D_{L}^{\alpha} {_{0}}D_{t}^{\alpha}u  = \nabla W(t, u),\quad t\in (0, L),\\
& u(0) = u(L) = 0.\nonumber
\end{eqnarray}
\end{Thm}

\begin{Remark}\label{openquestion}
{\rm
We recall that, Theorem \ref{Thm:MainTheorem1} and Theorem \ref{Thm:MainTheorem2}, give a positive answer to the question formulate in \cite{ZhangTorres}.

\noindent 
For the proof of Theorems \ref{Thm:MainTheorem1} and \ref{Thm:MainTheorem2} we adapt some ideas of \cite{MDLTJWFZ, ASTW, ZhangTorres}.}
\end{Remark}

The remaining part of this paper is organized as follows. Some preliminary results are presented in Section 2. In Section
3, we are devoted to accomplishing the proof of Theorem \ref{Thm:MainTheorem1} and in Section 4 we present the proof of Theorem \ref{Thm:MainTheorem2}.

\section{Preliminary Results}

In this section, for the reader's convenience, firstly we introduce some basic definitions of fractional calculus.
The Liouville-Weyl fractional integrals of order $0<\alpha<1$ are defined as
$$
_{-\infty}I^{\alpha}_x u(x)=\frac{1}{\Gamma(\alpha)}\int^x_{-\infty} (x-\xi)^{\alpha-1}u(\xi)d\xi
$$
and
$$
_{x}I^{\alpha}_{\infty} u(x)=\frac{1}{\Gamma(\alpha)}\int^{\infty}_{x}(\xi-x)^{\alpha-1}u(\xi)d\xi.
$$
The Liouville-Weyl fractional derivative of order $0<\alpha<1$ are defined as the left-inverse operators of the corresponding
Liouville-Weyl fractional integrals
\begin{equation}\label{eqn:RD}
_{-\infty}D^{\alpha}_x u(x)=\frac{d}{dx} {_{-\infty}I^{1-\alpha}_x u(x)}
\end{equation}
and
\begin{equation}\label{eqn:LD}
_{x}D^{\alpha}_{\infty} u(x)=-\frac{d}{dx} {_{x}I^{1-\alpha}_{\infty} u(x)}.
\end{equation}
The definitions of (\ref{eqn:RD}) and (\ref{eqn:LD}) may be written in an alternative form as follows:
$$
_{-\infty}D^{\alpha}_x u(x)=\frac{\alpha}{\Gamma(1-\alpha)}\int^{\infty}_0 \frac{u(x)-u(x-\xi)}{\xi^{\alpha+1}}d\xi
$$
and
$$
_{x}D^{\alpha}_{\infty} u(x)=\frac{\alpha}{\Gamma(1-\alpha)}\int^{\infty}_0 \frac{u(x)-u(x+\xi)}{\xi^{\alpha+1}}d\xi.
$$
Moreover, recall that the Fourier transform $\widehat{u}(w)$ of $u(x)$ is defined by
$$
\widehat{u}(w)=\int_{-\infty}^{\infty}e^{-iwx}u(x)dx.
$$

In order to establish the variational structure which enables us to reduce the existence of solutions of (FHS)$_\lambda$ to find critical points of the
corresponding functional, it is necessary to construct appropriate function spaces. In what follows, we introduce some fractional spaces, for more
details see \cite{ErvinR06}. To this end, denote by $L^p(\mathbb{R},\mathbb{R}^n)$ ($2\leq p <\infty$) the Banach spaces of functions on $\mathbb{R}$ with values in $\mathbb{R}^n$ under
the norms
$$
\|u\|_{L^p}=\Bigl(\int_{\mathbb{R}}|u(t)|^p dt\Bigr)^{1/p},
$$
and $L^{\infty}(\mathbb{R},\mathbb{R}^n)$ is the Banach space of essentially bounded functions from $\mathbb{R}$ into $\mathbb{R}^n$ equipped with the norm
$$
\|u\|_{\infty}=\mbox{ess} \sup\left\{|u(t)|: t\in \mathbb{R} \right\}.
$$
Let $0< \alpha \leq 1$ and $1<p<\infty$. The fractional derivative space $E_{0}^{\alpha ,p}$ is defined by the closure of $C_{0}^{\infty}([0,T], \mathbb{R}^n)$ with respect to the norm
\begin{equation}\label{norm}
\|u\|_{\alpha ,p} = \left(\int_{0}^{T} |u(t)|^pdt + \int_{0}^{T}|{_{0}}D_{t}^{\alpha}u(t)|^pdt  \right)^{1/p}, \;\;\forall\; u\in E_{0}^{\alpha ,p}.
\end{equation}
This space can be characterized by $E_{0}^{\alpha , p} = \{u\in L^{p}([0,T], \mathbb{R}^n)/\;\; {_{0}}D_{t}^{\alpha}u \in L^{p}([0,T], \mathbb{R}^n)\;\mbox{and}\;u(0) = u(T) = 0\}$. Moreover $(E_{0}^{\alpha ,p}, \|.\|_{\alpha ,p})$ is a reflexive and separable Banach space. Considering the space $E_{0}^{\alpha,p}$, we have the following results 

\begin{Prop}\label{FC-FEprop3}
\cite{FJYZ} Let $0< \alpha \leq 1$ and $1 < p < \infty$. For all $u\in E_{0}^{\alpha ,p}$, if $\alpha > 1/p$ we have
\begin{equation}\label{FC-FEeq3}
\|u\|_{L^{p}} \leq \frac{T^{\alpha}}{\Gamma (\alpha +1)} \|_{0}D_{t}^{\alpha}u\|_{L^{p}}.
\end{equation}
If $\alpha > 1/p$ and $\frac{1}{p} + \frac{1}{q} = 1$, then
\begin{equation}\label{FC-FEeq4}
\|u\|_{\infty} \leq \frac{T^{\alpha -1/p}}{\Gamma (\alpha)((\alpha - 1)q +1)^{1/q}}\|_{0}D_{t}^{\alpha}u\|_{L^{p}}.
\end{equation}
\end{Prop}

\noindent 
According to (\ref{FC-FEeq3}), we can consider in $E_{0}^{\alpha ,p}$ the following norm
\begin{equation}\label{FC-FEeq5}
\|u\|_{\alpha ,p} = \|_{0}D_{t}^{\alpha}u\|_{L^{p}},
\end{equation}
and (\ref{FC-FEeq5}) is equivalent to (\ref{norm}).

\begin{Prop}\label{FC-FEprop4}
\cite{FJYZ} Let $0< \alpha \leq 1$ and $1 < p < \infty$. Assume that $\alpha > \frac{1}{p}$ and $\{u_{k}\} \rightharpoonup u$ in $E_{0}^{\alpha ,p}$. Then $u_{k} \to u$ in $C[0,T]$, i.e.
$$
\|u_{k} - u\|_{\infty} \to 0,\;k\to \infty.
$$
\end{Prop}

We denote by $E^{\alpha} = E_{0}^{\alpha ,2}$, this is a Hilbert space with respect to the norm $\|u\|_{\alpha} = \|u\|_{\alpha ,2}$ given by (\ref{FC-FEeq5}).

For $\alpha>0$, define the semi-norm
$$
|u|_{I^{\alpha}_{-\infty}}=\|_{-\infty}D^{\alpha}_x u\|_{L^2}
$$
and the norm
\begin{equation}\label{eqn:defn Rnorm}
\|u\|_{I^{\alpha}_{-\infty}}=\Bigl(\|u\|^2_{L^2}+|u|^2_{I^{\alpha}_{-\infty}}\Bigr)^{1/2}
\end{equation}
and let
$$
I^{\alpha}_{-\infty}=\overline{C^{\infty}_0(\mathbb{R},\mathbb{R}^n)}^{\|\cdot\|_{I^{\alpha}_{-\infty}}},
$$
where $C_0^{\infty}(\mathbb{R},\mathbb{R}^n)$ denotes the space of infinitely differentiable functions from $\mathbb{R}$ into $\mathbb{R}^n$ with
vanishing property at infinity.

Now we can define the fractional Sobolev space $H^{\alpha}(\mathbb{R},\mathbb{R}^n)$ in terms of the Fourier transform. Choose $0<\alpha<1$, define
the semi-norm
$$
|u|_{\alpha}=\||w|^{\alpha}\widehat{u}\|_{L^2}
$$
and the norm
$$
\|u\|_{\alpha}=\Bigl(\|u\|^2_{L^2}+|u|^2_{\alpha}\Bigr)^{1/2}
$$
and let
\begin{equation}\label{eqn:alphanorm}
H^{\alpha}=\overline{C^{\infty}_0(\mathbb{R},\mathbb{R}^n)}^{\|\cdot\|_{\alpha}}.
\end{equation}
Moreover, we note that a function $u\in L^2(\mathbb{R},\mathbb{R}^n)$ belongs to $I^{\alpha}_{-\infty}$ if and only if
$$
|w|^{\alpha}\widehat{u}\in L^2(\mathbb{R},\mathbb{R}^n).
$$
Especially, we have
$$
|u|_{I^{\alpha}_{-\infty}}=\||w|\widehat{u}\|_{L^2}.
$$
Therefore, $I^{\alpha}_{-\infty}$ and $H^{\alpha}$ are equivalent with equivalent semi-norm and norm. Analogous to $I^{\alpha}_{-\infty}$,
we introduce $I^{\alpha}_{\infty}$. Define the semi-norm
$$
|u|_{I^{\alpha}_{\infty}}=\|_{x}D^{\alpha}_{\infty} u\|_{L^2}
$$
and the norm
\begin{equation}\label{eqn:defn Rnorm}
\|u\|_{I^{\alpha}_{\infty}}=\Bigl(\|u\|^2_{L^2}+|u|^2_{I^{\alpha}_{\infty}}\Bigr)^{1/2}
\end{equation}
and let
$$
I^{\alpha}_{\infty}=\overline{C^{\infty}_0(\mathbb{R},\mathbb{R}^n)}^{\|\cdot\|_{I^{\alpha}_{\infty}}}.
$$
Then $I^{\alpha}_{-\infty}$ and $I^{\alpha}_{\infty}$ are equivalent with equivalent semi-norm and norm, see \cite{ErvinR06}.

Let $C(\mathbb{R},\mathbb{R}^n)$ denote the space of continuous functions from $\mathbb{R}$ into $\mathbb{R}^n$. Then we obtain the following lemma.
\begin{lemma}\label{Lem:LinftyContH}\cite[Theorem 2.1]{Torres12}
If $\alpha>1/2$, then $H^{\alpha}\subset C(\mathbb{R},\mathbb{R}^n)$ and there is a constant $C_\infty=C_{\alpha,\infty}$ such that
\begin{equation}\label{12}
\|u\|_{\infty}=\sup_{x\in \mathbb{R}}|u(x)|\leq C_\infty \|u\|_{\alpha}.
\end{equation}
\end{lemma}

\begin{remark}\label{Rem:Lp}
From Lemma \ref{Lem:LinftyContH}, we know that if $u\in H^{\alpha}$ with $1/2<\alpha<1$, then $u\in L^p(\mathbb{R},\mathbb{R}^n)$ for all $p\in [2,\infty)$, since
$$
\int_{\mathbb{R}}|u(x)|^p dx \leq \|u\|^{p-2}_{\infty}\|u\|^2_{L^2}.
$$
\end{remark}

In what follows, we introduce the fractional space in which we will construct the variational framework of (FHS)$_\lambda$. Let
$$
X^{\alpha}=\Bigl\{u\in H^{\alpha}: \int_{\mathbb{R}}[|_{-\infty}D^{\alpha}_{t}u(t)|^2+(L(t)u(t),u(t))]dt<\infty\Bigr\},
$$
then $X^{\alpha}$ is a reflexive and separable Hilbert space with the inner product
$$
\langle u,v \rangle_{X^{\alpha}}=\int_{\mathbb{R}}[(_{-\infty}D^{\alpha}_{t}u(t),_{-\infty}D^{\alpha}_{t}v(t))+(L(t)u(t),v(t))]dt
$$
and the corresponding norm is
$$
\|u\|^2_{X^{\alpha}}=\langle u,u \rangle_{X^{\alpha}}.
$$
For $\lambda>0$, we also need the following inner product
$$
\langle u,v \rangle_{X^{\alpha,\lambda}}=\int_{\mathbb{R}}[(_{-\infty}D^{\alpha}_{t}u(t),_{-\infty}D^{\alpha}_{t}v(t))+\lambda(L(t)u(t),v(t))]dt
$$
and the corresponding norm is
$$
\|u\|^2_{X^{\alpha,\lambda}}=\langle u,u \rangle_{X^{\alpha,\lambda}}.
$$

\begin{lemma}\label{Lem:XcontH}
\cite{ZhangTorres} Suppose $L(t)$ satisfies {\rm ($\mathcal{L}$)$_1$} and {\rm ($\mathcal{L}$)$_2$}, then $X^{\alpha}$ is continuously embedded in $H^{\alpha}$.
\end{lemma}

\begin{remark}\label{keynta}
{\rm
Under the same conditions of Lemma \ref{Lem:XcontH}, for all
$\lambda\geq \frac{1}{c C_\infty^2 \, meas \{l<c\}}$, we also obtain
\begin{equation}\label{13}
\int_{\mathbb{R}}|u(t)|^2 dt\leq \frac{C_\infty^2\, meas\{l<c\}}{1-C_\infty^2\, meas\{l<c\}}\|u\|_{X^{\alpha,\lambda}}=\frac{1}{\Theta}\|u\|_{X^{\alpha,\lambda}}^2
\end{equation}
and
\begin{equation}\label{14}
\|u\|_\alpha^2\leq \Bigl(1+\frac{C_{\infty}^2\, meas\{l<c\}}{1-C_{\infty}^2\, meas \{l<c\}}\Bigr)\|u\|_{X^\alpha}^2=(1+\frac{1}{\Theta})\|u\|^2_{X^{\alpha,\lambda}}.
\end{equation}
Furthermore, for every $p\in (2,\infty)$ and $\lambda\geq \frac{1}{c C_\infty^2 \, meas \{l<c\}}$, we have
\begin{equation}\label{15}
\begin{split}
\int_{\mathbb{R}}|u(t)|^p dt\leq \frac{1}{\Theta^{\frac{p}{2}}\, (meas\{l<c\})^{\frac{p-2}{2}}}\|u\|_{X^{\alpha,\lambda}}^p.
\end{split}
\end{equation}
For more detail see \cite{ZhangTorres}.}
\end{remark}

\section{ Proof of Theorem \ref{Thm:MainTheorem1}}

The aim of section is to establish the proof of Theorem \ref{Thm:MainTheorem1}. For this purpose, we are going to establish the corresponding
variational framework to obtain solutions of (FHS)$_\lambda$. To this end, define the functional $I:\mathcal{B}=X^{\alpha,\lambda}\rightarrow \mathbb{R}$ by
\begin{equation}\label{mt01}
\begin{aligned}
I_\lambda (u)&=\int_{\mathbb{R}}\Bigl[\frac{1}{2}|_{-\infty}D_t^{\alpha}u(t)|^2+\frac{1}{2}(\lambda L(t)u(t),u(t))-W(t,u(t))\Bigr]dt\\
&=\dfrac{1}{2}\|u\|^2_{X^{\alpha,\lambda}}-\int_{\mathbb{R}}W(t,u(t))dt.
\end{aligned}
\end{equation}

Under the conditions of Theorem \ref{Thm:MainTheorem1}, as usual, we see that $I\in C^1(X^{\alpha,\lambda},\mathbb{R})$, i.e., $I$ is a continuously Fr$\acute{e}$chet-differentiable functional
defined on $X^{\alpha,\lambda}$. Moreover, we have
\begin{equation}\label{mt02}
I'_\lambda (u)v=\int_{\mathbb{R}}\Bigl[(_{-\infty}D_t^{\alpha}u(t), _{-\infty}D_t^{\alpha}v(t))+(\lambda L(t)u(t),v(t))-(\nabla W(t,u(t)),v(t))\Bigr]dt
\end{equation}
for all $u$, $v\in X^{\alpha}$, which yields that
\begin{equation}\label{mt03}
I'_\lambda (u)u =\|u\|^2_{X^{\alpha,\lambda}}-\int_{\mathbb{R}}(\nabla W(t,u(t)),u(t))dt.
\end{equation}

\begin{Remark}\label{mainnta1}
We note that $I_\lambda$ has the geometry property of Mountain Pass Theorem. In fact, first we prove that, there exist $\rho , \beta >0$ such that $I_{\lambda} |_{\partial B_{\rho}} \geq \beta$. By Remark \ref{keynta} and Lemma \ref{Lem:LinftyContH}, we have
$$
\|u\|_{L^2}^{2} \leq \frac{1}{\Theta}\|u\|_{X^{\alpha, \lambda}}^{2},\;\;\|u\|_{\alpha}^{2} \leq \left( 1+\frac{1}{\Theta} \right)\|u\|_{X^{\alpha, \lambda}}^{2}\:\:\mbox{and}\;\;\|u\|_{\infty} \leq C_{\infty} \|u\|_{\alpha}.
$$
Therefore 
\begin{equation}\label{mt04}
\|u\|_{\infty} \leq C_{\infty}\left( 1+\frac{1}{\Theta}\right)^{1/2}\|u\|_{X^{\alpha, \lambda}}.
\end{equation}
Now choose $\epsilon >0$ sufficiently small such that $\frac{1}{2} - \frac{\epsilon}{\Theta} >0$. By ($W_2$), $|W(t,u)| = o(|u|^2)$ uniformly in $t$ as $|u| \to 0$, then for all $\epsilon >0$, there exist $\delta >0$ such that
$$
|W(t,u(t))| \leq \epsilon |u(t)|^2\;\;\mbox{whenever} \;\;|u(t)| < \delta.
$$
Let $\rho = \frac{\delta}{C_{\infty}\left( 1+\frac{1}{\Theta} \right)^{1/2}}$ and $\|u\|_{X^{\alpha, \lambda}} \leq \rho$, then 
$$
|u(t)| \leq C_{\infty}\left( 1 + \frac{1}{\Theta} \right)^{1/2}\|u\|_{X^{\alpha , \lambda}} \leq \delta.
$$
Hence
\begin{equation}\label{mt05}
|W(t,u(t))| \leq \epsilon |u(t)|^2,\;\;\forall \;t\in \mathbb{R}.
\end{equation}
So, if $\|u\|_{X^{\alpha, \lambda}} = \rho$, then 
\begin{equation}\label{mt06}
\begin{aligned}
I_{\lambda}(u) &= \frac{1}{2}\|u\|_{X^{\alpha, \lambda}}^2 - \int_{\mathbb{R}}W(t,u(t))dt\\
&\geq \left( \frac{1}{2} - \frac{\epsilon}{\Theta} \right)\|u\|_{X^{\alpha, \lambda}}^{2}\\
&\geq \left( \frac{1}{2} - \frac{\epsilon}{\Theta} \right)\rho^2 \equiv \beta >0.
\end{aligned}
\end{equation}

Let $\varphi \in C_{0}^{\infty}(\mathbb{R},  \mathbb{R}^n)$ with $\|\varphi\|_{X^{\alpha, \lambda}} = 1$. It remains to prove that there exists an $e\in X^{\alpha, \lambda}$ such that $\|e\|_{X^{\alpha, \lambda}} >\rho$ and $I_{\lambda}(e) \leq 0$, where $\rho$ is defined above. Arguing by contradiction, we may assume that there exists $\{\sigma _k\} \subset \mathbb{R}$, $|\sigma_k| \to \infty$ such that $I_{\lambda}(\sigma_k \varphi) >0$ for all $k$. Then, we have
\begin{equation}\label{mt07}
0< \frac{I_{\lambda}(\sigma_k \varphi)}{\sigma_k^2} = \frac{1}{2} - \int_{\mathbb{R}} \frac{W(t, \sigma_k \varphi)}{|\sigma_k \varphi|^2} |\varphi|^2dt.
\end{equation}   
Since $|\sigma_k \varphi(t)| \to \infty$ for $t$ with $\varphi (t) \neq 0$, and since $\|\varphi\|_{X^{\alpha, \lambda}} = 1$, by (\ref{eqif}) and Fatou's Lemma, we have that
$$
\int_{\mathbb{R}} \frac{W(t, \sigma_k \varphi)}{|\sigma_k \varphi|^2}|\varphi|^2dt \to \infty\;\;\mbox{as}\;\;k\to \infty.
$$
This contradicts (\ref{mt07}). So we conclude taking $e = \sigma \varphi$ with $\sigma$ large enough.
\end{Remark}

Now, let us introduce the Nehari's manifold defined by
$$
\mathcal{N}_{\lambda} = \{u\in X^{\alpha, \lambda}\setminus \{0\}:\;\; \langle I'_{\lambda}(u), u \rangle = 0\},
$$ 
and we note that, for $u\in \mathcal{N}_{\lambda}$
$$
\begin{aligned}
I_{\lambda}(u) &= I_{\lambda}(u) - \frac{1}{2}\langle I'_{\lambda}(u), u \rangle= \int_{\mathbb{R}} \left(\frac{1}{2}( \nabla W(t, u(t)), u(t)) - W(t,u(t))   \right)dt. 
\end{aligned}
$$
Define
$$
c_{\lambda} = \inf_{\mathcal{N}_\lambda} I_{\lambda}(u).
$$

In the following Lemmas we assume that $(\mathcal{L}_1)-(\mathcal{L}_2)$, $(W_1)-(W_4)$ hold and $\lambda > 0$

\begin{lemma}\label{TMlem1}
Let $S_{\lambda} = \{u\in X^{\alpha, \lambda}: \;\;\|u\|_{X^{\alpha, \lambda}} = 1\}$. For all $u \in S_{\lambda}$ there exists a unique $\sigma_{u} >0$ such that $\sigma_u u \in \mathcal{N}_{\lambda}$. Furthermore
$$
I_{\lambda}(\sigma_u u) = \max_{\sigma \geq 0} I_{\lambda}(\sigma u)
$$
\end{lemma}
\begin{proof}
Let $u\in S_{\lambda}$ be fixed and define $h(\sigma) = I_{\lambda}(\sigma u)$ for $\sigma \geq 0$. Then
\begin{equation}\label{mt08}
\begin{aligned}
h(\sigma) &= \frac{\sigma^2}{2}\|u\|_{X^{\alpha, \lambda}}^2 - \int_{\mathbb{R}}W(t, \sigma u(t))dt\\
&= \sigma^{2}\left( \frac{1}{2} - \int_{\mathbb{R}}\frac{W(t,\sigma u(t))}{\sigma^2}dt \right).
\end{aligned}
\end{equation}
By ($W_2$), as $\sigma \to 0$ 
\begin{equation}\label{mt09}
\int_{\mathbb{R}} \frac{W(t, \sigma u(t))}{\sigma^2}dt \to 0
\end{equation}
and by (\ref{eqif}), as $\sigma \to \infty$,
\begin{equation}\label{mt10}
\int_{\mathbb{R}} \frac{W(t, \sigma u(t))}{\sigma^2}dt \to \infty.
\end{equation} 
Consequently, by (${W}_4$) and (\ref{mt08})-(\ref{mt10}), there is a unique $\sigma_u = \sigma(u)>0$ such that $h'(\sigma_u) = 0$ and
\begin{equation}\label{mt11}
h(\sigma_u) = \max_{\sigma \geq 0} I_{\lambda}(\sigma u).
\end{equation}  
Furthermore $\sigma_u u \in \mathcal{N}_{\lambda}$.
\end{proof}

\begin{lemma}\label{TMlem2}
The set $\mathcal{N}_{\lambda}$ is bounded away from 0. Furthermore, $\mathcal{N}_\lambda$ is closed in $X^{\alpha, \lambda}$.
\end{lemma}
\begin{proof}
Following the same way of Remark \ref{mainnta1}, we can conclude that
\begin{equation}\label{mt12}
I_{\lambda}(u) = \frac{1}{2}\|u\|_{X^{\alpha, \lambda}}^{2} + o(\|u\|_{X^{\alpha, \lambda}}^{2})\;\;\mbox{as}\;\;u \to 0.
\end{equation} 
Therefore there exists $\nu >0$ such that $u\in \mathcal{N}_{\lambda}$ implies $\|u\|_{X^{\alpha, \lambda}} \geq \nu$. So, $\mathcal{N}_{\lambda}$ is bounded away from $0$.

Now we prove that the set $\mathcal{N}_{\lambda}$ is closed in $X^{\alpha, \lambda}$. First, we note that $I'_\lambda$ maps bounded sets in $X^{\alpha, \lambda}$ into bounded sets in $X^{\alpha, \lambda}$. In fact, let $\{u_k\}$ be a bounded sequence in $X^{\alpha, \lambda}$, then by (\ref{12}) and (\ref{14}), there exists $K_1>0$ such that for each $k\in \mathbb{N}$:   
$$
\|u_k\|_{\infty} \leq K_1.
$$
From $(W_2)$, there exists $\delta >0$ such that for all $t\in \mathbb{R}$ and $|u|< \delta$
$$
|\nabla W(t,u)| \leq |u|.
$$
Now, let $M_1 = \max\{\overline{W}(u)/\;\;|u|\leq K_1\}$ and $K_2 = \max\{1, \frac{M_1}{\delta}\}$. If $|u_k(t)| < \delta$, then 
$$
|\nabla W(t,u_k(t))| \leq |u_k(t)|.
$$
On the other hand, by ($W_3$), if $\delta \leq |u_k(t)| \leq K_1$, then 
$$
|\nabla W(t, u_k(t))| \leq \overline{W}(u_k(t)) \leq M_1 \leq \frac{M_1}{\delta}|u_k(t)|.
$$
Therefore, for all $k\in \mathbb{N}$ and $t\in \mathbb{R}$
\begin{equation}\label{mt13}
|\nabla W(t, u_k(t))| \leq K_2 |u_k(t)|.
\end{equation}
Next, by (\ref{mt13}), H\"older inequality and (\ref{13})
$$
\begin{aligned}
\left| \int_{\mathbb{R}^n} ( \nabla W(t, u_k(t)), \varphi (t)) dt  \right| &\leq K_2\int_{\mathbb{R}} |u_k(t)||\varphi (t)|dt \\
& \leq \frac{K_2}{\Theta} \|u_k\|_{X^{\alpha , \lambda}} \|\varphi \|_{X^{\alpha , \lambda}}\quad \forall \varphi \in X^{\alpha, \lambda}.
\end{aligned}
$$
So, for each $\varphi \in X^{\alpha, \lambda}$
$$
\begin{aligned}
I'_\lambda (u_k)\varphi &= \langle u_k, \varphi\rangle_{X^{\alpha, \lambda}} - \int_{\mathbb{R}}( \nabla W(t, u_k(t)), \varphi (t)) dt\\
&\leq \|u_k\|_{X^{\alpha, \lambda}}^2\|\varphi\|_{X^{\alpha, \lambda}}^{2} + \frac{K_2}{\Theta}\|u_k\|_{X^{\alpha, \lambda}} \|\varphi\|_{X^{\alpha, \lambda}} \leq K_3.
\end{aligned}
$$
Now we are in position to prove that $\mathcal{N}_{\lambda}$ is closed in $X^{\alpha, \lambda}$. Let  $u_k\in \mathcal{N}_{\lambda}$ such that $u_k \to u$ in $X^{\alpha, \lambda}$. Since $I'_\lambda (u_k)$ is bounded, then we infer from
$$
I'_\lambda(u_k)u_k - I'_{\lambda}(u)u = \langle I'_\lambda (u_k) - I'_\lambda (u), u\rangle - \langle I'_\lambda (u_k), u_k - u\rangle \to 0, \;\;\mbox{as}\;\;k\to \infty,
$$ 
that $I'_\lambda (u) u = 0$. Furthermore, since $\mathcal{N}_\lambda$ is bounded away from $0$, we have 
$$
\|u\|_{X^{\alpha, \lambda}} = \lim_{k\to \infty} \|u_k\|_{X^{\alpha, \lambda}} \geq \nu>0.
$$
So $u\in \mathcal{N}_\lambda$.
\end{proof}

\begin{lemma}\label{TMlem3}
There exists $\kappa>0$ such that $\sigma_u \geq \kappa$ for all $u\in S_{\lambda}$, and for each compact subset $\mathfrak{W}\in S_{\lambda}$ there exists a constant $C_{\mathfrak{W}}>0$ such that 
$$
\sigma_u \leq C_{\mathfrak{W}}\quad \mbox{for all}\;\;u\in S_{\lambda}.
$$
\end{lemma}
\begin{proof}
For $u\in S_{\lambda}$, there exists $\sigma_u >0$ such that $\sigma_u u\in \mathcal{N}_{\lambda}$. By Lemma \ref{TMlem2}, one sees that $\sigma_u \geq \nu>0$. To prove that $\sigma_u \leq C_{\mathfrak{W}}$ for all $u\in \mathfrak{W} \subset S_{\lambda}$, we argue by contradiction. Suppose that there exists $u_k \in \mathfrak{W}$ such that $\sigma_k = \sigma_{u_k} \to \infty.$ Since $\mathfrak{W}$ is compact, there exists $u\in \mathfrak{W}$ such that $u_k \to u$ in $X^{\alpha, \lambda}$ and $u_k(t) \to u(t)$ a.e. on $\mathbb{R}$. Therefore 
\begin{equation}\label{mt14}
\begin{aligned}
\frac{I_{\lambda} (\sigma_k u_k)}{\sigma_k^2} &= \frac{1}{2}\|u_k\|_{X^{\alpha, \lambda}}^{2} - \int_{\mathbb{R}}\frac{W(t,\sigma_k u_k(t))}{\sigma_k^2}dt\\
&= \frac{1}{2} - \int_{\mathbb{R}} \frac{W(t, \sigma_ku_k(t))}{|\sigma_k u_k(t)|^2}|u_k(t)|^2dt.
\end{aligned}
\end{equation}  
Since $|\sigma_k u_k(t)| \to \infty$ if $u(t) \neq 0$, it follows from (\ref{eqif}), (\ref{mt14}) and Fatou's Lemma that $I_{\lambda}(\sigma_k u_k) \to -\infty$ as $k\to \infty$.  
\end{proof}

\begin{lemma}\label{TMlem4}
$c_\lambda \geq \rho >0$, where $\rho >0$ is independent of $\lambda$. 
\end{lemma}
\begin{proof}
For $u\in \mathcal{N}_{\lambda}$, ($W_1$) and Lemma \ref{TMlem2}, we obtain:
$$
\begin{aligned}
I_{\lambda}(u) &= I_\lambda(u) - \frac{1}{\theta} \langle I'_\lambda(u), u \rangle\\
& = \left( \frac{1}{2}-\frac{1}{\theta} \right)\|u\|_{X^{\alpha, \lambda}}^{2} + \int_{\mathbb{R}} \left(\frac{1}{\theta} ( \nabla W(t,u), u) - W(t,u)  \right)dt\\
&\geq \left( \frac{1}{2} - \frac{1}{\theta}\right)\|u\|_{X^{\alpha, \lambda}}^{2}\geq \left( \frac{1}{2}-\frac{1}{\theta} \right) \nu := \rho >0.
\end{aligned}
$$
\end{proof}

\begin{remark}\label{TMnta1}
Following the same way of \cite{Torres17}, by Lemma \ref{TMlem1} we can get the following characterization:
$$
c_{\lambda} = \inf_{u\in \mathcal{N}_{\lambda}} I_\lambda (u) = \inf_{u\in X^{\alpha, \lambda}\setminus \{0\}} \max_{s>0} I_\lambda (su) = \inf_{u\in S_\lambda \setminus \{0\}} \max_{s>0} I_\lambda (su).
$$ 
On the other hand, choosing $\varphi_0 \in C_{0}^{\infty}(T)$, there exists a constant $\mathfrak{C}_0 >0$ independent of $\lambda$, such that
\begin{equation}\label{mt31}
c_\lambda = \inf_{u\in X^{\alpha, \lambda}\setminus \{0\}} \max_{s >0} I_{\lambda}(s u) \leq \max_{s \geq 0} I_\lambda (s \varphi_0) \leq \mathfrak{C}_0.
\end{equation} 
\end{remark}

Now, define the mapping $m_\lambda : S_\lambda \to \mathcal{N_\lambda}$ by setting
$$
m_\lambda (u) :=\sigma_u u,
$$
where $\sigma_u$ is as in Lemma \ref{TMlem1} - 1 and $S_\lambda$ is the unit sphere in $X^{\alpha, \lambda}$. Furthermore, by Lemma \ref{TMlem1} and Proposition $3.1$ of \cite{ASTW}, $m_\lambda$ is a homeomorphism between $S_\lambda$ and $\mathcal{N}_\lambda$ and the inverse of $m_\lambda$ is given by
\begin{equation}\label{mt15}
m_{\lambda}^{-1}(u) = \frac{u}{\|u\|_{X^{\alpha, \lambda}}}.
\end{equation}

Now we shall consider the functional $\Phi_{\lambda}: S_\lambda \to \mathbb{R}$ defined by
$$
\Phi_\lambda (u) = I_{\lambda}(m_{\lambda}(u)).
$$
As in \cite{ASTW}, we have the following Lemma.

\begin{lemma}\label{TMlem5}
\begin{enumerate}
\item $\Phi_{\lambda} \in C^1(S_{\lambda}, \mathbb{R})$ and
$$
\langle \Phi'_\lambda (u), v \rangle = \|m_\lambda (u)\|_{X^{\alpha, \lambda}} \langle I'_\lambda (m_\lambda(u)), v\rangle  
$$
for all $v\in \mathcal{T}_{w}(S_\lambda) = \{h\in X^{\alpha, \lambda}:\;\; \langle u, h\rangle_{X^{\alpha, \lambda}} = 0 \}$
\item If $\{u_n\}$ is a PS sequence for $\Phi_\lambda$ then $\{m_\lambda (u_n)\}$ is a PS sequence for $I_\lambda$. If $\{u_n\} \subset \mathcal{N}_{\lambda}$ is a bounded PS sequence for $I_\lambda$, then $\{m_{\lambda}^{-1}(u_n)\}$ is a PS sequence for $\Phi_\lambda$, where $m_\lambda^{-1}(u)$ is given by (\ref{mt15}).
\item  $$
\inf_{S_\lambda} \Phi_\lambda = \inf_{\mathcal{N}_\lambda} I_\lambda.
$$  
Furthermore, $u$ is a critical point of $\Phi_\lambda$ if and only if $m_\lambda(u)$ is a nontrivial critical point of $I_\lambda$. Furthermore, the corresponding critical values of $\Phi_\lambda$ and $I_\lambda$ coincide.
\end{enumerate}
\end{lemma}

Now, we investigate the minimizing sequence for $I_\lambda$.

\begin{lemma}\label{TMlem6}
Suppose that $(\mathcal{L}_1)-(\mathcal{L}_2)$, $(W_1)-(W_4)$ hold and $\lambda \geq 1$. If $\{u_n\} \subset \mathcal{N}_\lambda$ be a minimizing sequence for $I_\lambda$, then $\{u_n\}$ is bounded in $X^{\alpha, \lambda}$. 
\end{lemma}

\begin{proof}
Let $\{u_n\} \subset \mathcal{N}_\lambda$ such that 
$$
I_\lambda (u_n) \to c_\lambda,\;\;\mbox{as}\;\;n\to \infty.
$$
Then, by ($W_1$) and $\langle I'_\lambda(u_n), u_n \rangle =0$, we obtain
\begin{equation}\label{mt16}
\begin{aligned}
c_\lambda + o(1) &= \left( \frac{1}{2}-\frac{1}{\theta} \right)\|u_n\|_{X^{\alpha, \lambda}}^{2} + \int_{\mathbb{R}} \left(\frac{1}{\theta}( \nabla W(t,u_n(t)), u_n(t)) - W(t,u_n(t))  \right)dt\\
&\geq \left( \frac{1}{2}-\frac{1}{\theta} \right)\|u_n\|_{X^{\alpha, \lambda}}^{2}.
\end{aligned}
\end{equation}
Therefore, (\ref{mt16}) implies that $\{u_n\}$ is bounded in $X^{\alpha, \lambda}$.
\end{proof}

Following the same way of Lemmas $2.1$ and $3.3$ in \cite{Torres17}, we can show the following version of the Lions concentration compactness principle.

\begin{lemma}\label{TMlemCC}
Let $r>0$ and $q\geq 2$. Let $\{u_n\} \in X^{\alpha, \lambda}$ be bounded. If
$$
\lim_{n\to \infty} \sup_{y\in \mathbb{R}} \int_{(y-r, y+r)} |u_n(t)|^qdt =0,
$$
then $u_n \to 0$ in $L^p(\mathbb{R}, \mathbb{R}^n)$ for any $p>2$.
\end{lemma}

\begin{lemma}\label{TMlem5}
Under the assumptions of Theorem \ref{Thm:MainTheorem1}, if $\{u_n\} \subset \mathcal{N}_\lambda$ be a sequence such that 
\begin{equation}\label{mt17}
I_\lambda(u_n) \to c_\lambda\;\;\mbox{and}\;\;I'_\lambda(u_n) \to 0,
\end{equation} 
then there exists $\Lambda >0$ such that $\{u_n\}$ has a convergent subsequence in $X^{\alpha, \lambda}$ for all $\lambda > \Lambda$.
\end{lemma}

\begin{proof}
By (\ref{mt16}) and (\ref{mt17}) we deduce that $\{u_n\}$ is bounded in $X^{\alpha, \lambda}$. Since $X^{\alpha, \lambda}$ is a reflexive space, there is a subsequence still called $\{u_n\} \in X^{\alpha, \lambda}$ and $u\in X^{\alpha, \lambda}$ such that $u_n \rightharpoonup u$. Furthermore, by Remark \ref{keynta} and Sobolev Theorem 
$$
u_n \to u \;\;\mbox{in}\;\;L_{loc}^p(\mathbb{R}) \;\;\mbox{for}\;\;p\in [2,\infty],
$$
and, we have either $\{u_n\}$ is vanishing, namely
\begin{equation}\label{mt18}
\lim_{n\to \infty} \sup_{t\in \mathbb{R}}\int_{(t-r,t+r)} |u_n(s)|^2ds=0
\end{equation}
or non-vanishing, namely, there exists $r, \beta >0$ and a sequence $\{t_n\} \subset \mathbb{R}$ such that
\begin{equation}\label{mt19}
\lim_{n\to \infty} \int_{(t_n-r,t_n+r)} |u_n(s)|^2ds\geq \beta.
\end{equation}  
We claim that $u\neq 0$. By contradiction, we suppose that $u=0$. If $\{u_n\}$ is vanishing, by Lemma \ref{TMlemCC}, $u_n \to 0$ in $L^p(\mathbb{R}, \mathbb{R}^n)$ for $p>2$. So, following the ideas of the proof of Lemma \ref{TMlem2} , we deduce that 
\begin{equation}\label{mt20}
\int_{\mathbb{R}} ( \nabla W(t,u_n(t)), u_n(t)) dt \to 0.
\end{equation}
Therefore, by (\ref{mt20}) and $\langle I'_\lambda (u_n) , u_n \rangle = 0$, we obtain that 
$$
\|u_n\|_{X^{\alpha, \lambda}} \to 0\;\;\mbox{as}\;\;n\to \infty.
$$
This contradicts the conclusion of Lemma \ref{TMlem2}.

On the other hand, by (\ref{mt31}) and (\ref{mt16}) we have
\begin{equation}\label{mt38}
\limsup_{n\to \infty} \|u_n\|_{X^{\alpha, \lambda}}^2 \leq \frac{2\theta}{\theta - 2} \mathfrak{C}_0.
\end{equation}
Therefore, if $\{u_n\}$ is non vanishing, then (\ref{mt19}) implies that $|t_n| \to \infty$ as $n\to \infty$. Then
$$|(t_n-r,t_n+r) \cap \{t\in \mathbb{R}:\;\;l(t)<c\}| \to 0\;\;\mbox{as}\;\;n\to \infty.$$ 
So, by H\"older inequality, we obtain
\begin{equation}\label{mt39}
\int_{(t_n-r, t_n+t) \cap \{l<c\}}u_n^2dt \to 0.
\end{equation} 
Combining (\ref{mt19}), (\ref{mt38}) and (\ref{mt39}), one has that
\begin{eqnarray}\label{mt40}
\frac{2\theta}{\theta-2} \mathfrak{C}_0 &\geq& \limsup_{n\to \infty} \|u_n\|_{X^{\alpha, \lambda}}^{2} \geq \lambda c \limsup_{n\to \infty} \int_{(t_n-r,t_n+r)\cap \{l\geq c\}} u_n^2(t)dt\nonumber\\
& = & \lambda c \limsup_{n\to \infty} \left( \int_{(t_n-r,t_n+r)} u_n^2(t)dt - \int_{(t_n-r,t_n+r)\cap \{l<c\}} u_{n}^{2}(t)dt \right)\nonumber\\
&\geq& \lambda c \beta.
\end{eqnarray}
Let $\Lambda_* = \max\{\frac{1}{cC_{\infty}^{2}meas\{l< c\}}, \frac{2\theta \mathfrak{C}_0}{(\theta-2)c\beta}\}$, then we obtain that $\lambda > \Lambda_* > \frac{2\theta\mathfrak{C}_0}{(\theta-2)c\beta}$, wich contradicts with (\ref{mt40}) .

To conclude, we need to prove that $u_k\to u$ in $X^{\alpha, \lambda}$. First, we note that the function $\frac{1}{\theta}( \nabla W(t,su), su) - W(t,su)$ is non-decreasing for $s>0$. In fact, let $0<s_1<s_2$, then we have
$$
\begin{aligned}
&(\nabla W(t,s_1u), s_1u) - \theta W(t, s_1u) \\
&= (\nabla W(t,s_1u), s_1u)+ \theta W(t, s_2u) - \theta W(t, s_2u)- \theta W(t, s_1u) \\
& = ( \nabla W(t,s_1u), s_1u) - \theta W(t, s_2u) + \theta \int_{s_1}^{s_2} (\nabla W(t, ru), u) dr\\
&\leq ( \nabla W(t,s_1u), s_1u) - \theta W(t, s_2u) + \frac{( \nabla W(t, s_2u), u)}{s_2^{\theta - 1}}(s_{2}^{\theta} - s_{1}^{\theta})\\
&\leq ( \nabla W(t,s_1u), s_1u) - \theta W(t, s_2u) + s_2( \nabla W(t, s_2u), u) - s_1( \nabla W(t,s_1u, u))\\
& = ( \nabla W(t,s_2u), s_2u) - \theta W(t, s_2u)
\end{aligned}
$$
Finally, since $u\neq 0$ and Lemma \ref{TMlem1} there exists $\sigma \in (0,1]$ such that $\sigma u \in \mathcal{N}_\lambda$, then by Fatou's Lemma , it is easy to check that
\begin{align*}
c_{\lambda} \leq & I_{\lambda}(\sigma u) = I_{\lambda}(\sigma u) - \frac{1}{\theta} I'_{\lambda}(\sigma u) \sigma u \\
&= \sigma^2\left(\frac{1}{2}-\frac{1}{\theta} \right)\|u\|_{X^{\alpha, \lambda}}^{2}+\int_{\mathbb{R}}(\frac{1}{\theta}( \nabla W(t,\sigma u(t)), \sigma u(t) ) - W(t,\sigma u(t)))dt\\
&\leq \left(\frac{1}{2}-\frac{1}{\theta} \right)\|u\|_{X^{\alpha, \lambda}}^{2}+\int_{\mathbb{R}}(\frac{1}{\theta}(\nabla W(t,u(t)), u(t) ) - W(t, u(t)))dt\\
&\leq  \liminf_{n \to \infty}\left\{ \left(\frac{1}{2}-\frac{1}{\theta} \right)\|u_n\|_{X^{\alpha, \lambda}}^{2}+\int_{\mathbb{R}}(\frac{1}{\theta}( \nabla W(t,u_n(t)), u_n(t) ) - W(t,u_n(t)))\right\}\\
&\leq  \limsup_{k \to \infty}\left\{ \left(\frac{1}{2}-\frac{1}{\theta} \right)\|u_n\|_{X^{\alpha, \lambda}}^{2}+\int_{\mathbb{R}}(\frac{1}{\theta}(\nabla W(t,u_n(t)), u_n(t) ) - W(t,u_n(t)))\right\}\\
&=  \lim_{n \to \infty} \left\{ I_{\lambda}(u_{n}) - \frac{1}{\theta} I'_{\lambda}(u_{n}) u_{n} \right\} = \, c_{\lambda} .\\
\end{align*}
Hence,
$$
\|u_n\|_{X^{\alpha, \lambda}}^2 \to \|u\|_{X^{\alpha, \lambda}}^2 \quad \mbox{in} \quad  \mathbb{R},
$$
from where it follows that $ u_{n} \to u $ in $X^{\alpha, \lambda}$.
\end{proof}

\section{Proof of Theorem \ref{Thm:MainTheorem2}}

In the following, we study the concentration of solutions for problem $(\mbox{FHS})_{\lambda}$ as $\lambda \to \infty$. Firstly, for technical reason we consider $T = [0, L]$ and the following fractional boundary value problem
\begin{equation}\label{eqn:BVP}
\left\{
  \begin{array}{ll}
   {_{t}}D_{L}^{\alpha} {_{0}}D_{t}^{\alpha}u  = \nabla W(t, u),\quad t\in (0, L),\\[0.1cm]
    u(0) = u(L) = 0.
  \end{array}
\right.
\end{equation}
Associated to (\ref{eqn:BVP}) we have the functional $I: E_{0}^{\alpha} \to \mathbb{R}$ given by
$$
I(u):= \frac{1}{2}\int_{0}^{T} |{_{0}}D_{t}^{\alpha}u(t)|^2dt - \int_{0}^{T}F(t,u(t))dt
$$
and we have that $I\in C^1(E_{0}^{\alpha}, \mathbb{R})$ with 
$$
I'(u)v = \int_{0}^{T} \langle {_{0}}D_{t}^{\alpha}u(t), {_{0}}D_{t}^{\alpha}v(t)\rangle dt - \int_{0}^{T}\langle\nabla W(t,u(t)), v(t) \rangle.
$$
The Nehari manifold corresponding to $I$ is defined by 
$$
\tilde{\mathcal{N}} = \{u\in E_{0}^{\alpha}\setminus \{0\}: \;\;I'(u)u = 0\},
$$  
and let 
$$
\tilde{c} = \inf_{u\in \tilde{\mathcal{N}}}I(u).
$$
Furthermore, we can show that
$$\tilde{c} =  \inf_{w\in E_{0}^{\alpha}} \max_{\sigma>0}I(\sigma w) = \inf_{u\in \tilde{S} \setminus \{0\}} \max_{\sigma >0}I(\sigma u),$$
and if we follow the ideas of the proof of Theorem \ref{Thm:MainTheorem1}, we can get the following existence result
\begin{Thm}\label{Gthm1}
Suppose that $W$ satisfies $(W_1)-(W_4)$ with $t\in [0,L]$, then (\ref{eqn:BVP}) has a ground state solution.
\end{Thm}
Furthermore, under the assumptions $(\mathcal{L})_1-(\mathcal{L})_3$ and $(W_1)-(W_2)$, we can get that
$$
c_\lambda \leq \tilde{c}\;\;\mbox{for}\;\;\lambda >0.
$$
In fact, by Theorem \ref{Gthm1}, let $\tilde{u} \in E_{0}^{\alpha}$ be a ground state solution of (\ref{eqn:BVP}), then $\tilde{c} = I(\tilde{u})$. Therefore,
$$
c_\lambda \leq \max_{\sigma >0} I_\lambda(\sigma \tilde{u}) = \max_{\sigma >0} I(\sigma \tilde{u}) = I(\tilde{u}) = \tilde{c}\;\;\mbox{for all}\;\;\lambda >0.
$$

\noindent
\begin{proof}{\bf Theorem \ref{Thm:MainTheorem2}} We follow the argument in \cite{ZhangTorres}.
For any sequence $\lambda_k \to \infty$, let $u_k = u_{\lambda_k}$ be the critical point of $I_{\lambda_k}$, namely
$$
c_{\lambda_k} = I_{\lambda_k}(u_k)\quad \mbox{and}\quad I'_{\lambda_k}(u_k)=0,
$$
and by ($W_1$) we get
$$
\begin{aligned}
c_{\lambda_k} &= I_{\lambda_k}(u_k) = I_{\lambda_k}(u_k) - \frac{1}{\theta} I'_{\lambda_k}(u_k)u_k\\
&=\left( \frac{1}{2} - \frac{1}{\theta} \right)\|u_k\|_{X^{\alpha , \lambda_k}}^{2} + \int_{\mathbb{R}} \left[\frac{1}{\theta}( \nabla W(t,u_k(t)), u_k(t)) - W(t,u_k(t))\right]dt\\
&\geq \left( \frac{1}{2} - \frac{1}{\theta} \right)\|u_k\|_{X^{\alpha ,\lambda_k}}^{2}
\end{aligned}
$$
Therefore, by (\ref{mt31})
\begin{equation}\label{esti}
\sup_{k\geq 1} \|u_k\|_{X^{\alpha, \lambda_k}}^{2}\leq \frac{2\theta}{\theta - 2}\mathfrak{C}_0,
\end{equation}
where $\mathfrak{C}_0$ is independent of $\lambda_k$. Therefore, we may assume that $u_k \rightharpoonup \tilde{u}$ weakly in $X^{\alpha,\lambda_k}$.
Moreover, by Fatou's lemma, we have
$$
\begin{aligned}
\int_{\mathbb{R}} l(t) |\tilde{u}(t)|^2dt \leq& \liminf_{k\to \infty} \int_{\mathbb{R}} l(t)|u_k(t)|^2dt\\
\leq& \liminf_{k\to \infty} \int_{\mathbb{R}} (L(t)u_k(t), u_{k}(t)) dt\\
\leq& \liminf_{k\to \infty} \frac{\|u_k\|_{X^{\alpha, \lambda_k}}^2}{\lambda_{k}} = 0.
\end{aligned}
$$
Thus, $\tilde{u} = 0$ a.e. in $\mathbb{R} \setminus J$. Now, for any $\varphi \in C_{0}^{\infty}(T, \mathbb{R}^n)$, since $I'_{\lambda_k}(u_k)\varphi=0$, it is 
easy to see that
$$
\int_{0}^{L} ({_{0}}D_{t}^{\alpha} \tilde{u}(t), {_{0}}D_{t}^{\alpha}\varphi(t)) dt - \int_{0}^{L} (\nabla W(t,\tilde{u}(t)), \varphi (t)) dt=0,
$$
that is, $\tilde{u}$ is a solution of (\ref{eqn:BVP}) by the density of $C_{0}^{\infty}(T, \mathbb{R}^n)$ in $E^{\alpha}$.

Next, we show that $u_k \to \tilde{u}$ strongly in $L^{r}(\mathbb{R})$ for
$2\leq r < \infty$. Otherwise, by Lemma \ref{TMlemCC}, there exist $\delta>0$, $R_0>0$ and $t_k\in \R$ such that
$$
\int_{t_k-R_0}^{t_k+R_0}(u_k-\tilde{u})^2dt\geq \delta.
$$
Moreover, $t_n\to \infty$, hence $meas\{(t_k-R_0,t_k+R_0)\cap \{l<c\}\}\to 0$. By the H\"{o}lder inequality, we have
$$
\int_{(t_k-R_0,t_k+R_0)\cap \{l<c\}}|u_k-\tilde{u}|^2dt\leq meas \{(t_k-R_0,t_k+R_0)\cap \{l<c\}\}\|u_k-\tilde{u}\|_\infty\to 0.
$$
Consequently,
$$
\begin{aligned}
\|u_k\|_{X^{\alpha,\lambda_k}}^2&\geq \lambda_k  c\int_{(t_k-R_0,t_k+R_0)\cap \{l\geq c\}} |u_k(t)|^2 dt\\
&=\lambda_k c \int_{(t_k-R_0,t_k+R_0)\cap \{l\geq c\}} |u_k(t)-\tilde{u}(t)|^2 dt\\
&= \lambda_k c \Bigl(\int_{(t_k-R_0,t_k+R_0)}|u_k(t)-\tilde{u}(t)|^2 dt-\int_{(t_k-R_0,t_k+R_0)\cap \{l<c\}}|u_k-\tilde{u}|^2dt\Bigr)+o(1)\\
&\to \infty,
\end{aligned}
$$
which contradicts (\ref{esti}).

Now we show that $u_k \to \tilde{u}$ in $X^{\alpha}$. Since $I'_{\lambda_k}(u_k) u_k=I'_{\lambda_k}(u_k)\tilde{u}=0$, we have
\begin{equation}\label{c6}
\|u_k\|_{X^{\alpha, \lambda_k}}^{2} = \int_{\mathbb{R}} (\nabla W(t, u_k(t)), u_k(t)) dt
\end{equation}
and
\begin{equation}\label{c7}
\langle u_k, \tilde{u} \rangle_{\lambda_k} = \int_{\mathbb{R}} (\nabla W(t, u_k(t)), \tilde{u}(t)) dt,
\end{equation}
which implies that
$$
\lim_{k\to \infty} \|u_{k}\|_{X^{\alpha, \lambda_k}}^{2} = \lim_{k\to \infty} \langle u_k, \tilde{u} \rangle_{X^{\alpha, \lambda_k}} = \lim_{k\to \infty} \langle u_k, \tilde{u} \rangle_{X^{\alpha}} = \|\tilde{u}\|_{X^\alpha}^2.
$$
Furthermore, by the weakly semi-continuity of norms we obtain
$$
\|\tilde{u}\|_{X^{\alpha}}^{2} \leq \liminf_{k \to \infty} \|u_k\|_{X^{\alpha}}^{2} \leq \limsup_{k\to \infty}\|u_k\|_{X^\alpha}^{2} \leq \lim_{k\to \infty}\|u_k\|_{X^{\alpha,\lambda_k}}^{2}.
$$
So $u_k \to \tilde{u}$ in $X^{\alpha}$, and $u_k \to \tilde{u}$ in $H^{\alpha}(\mathbb{R}, \mathbb{R}^n)$ as $k\to \infty$. 
\end{proof}

\end{document}